\documentclass[12pt]{amsart}
\usepackage[utf8]{inputenc}
\usepackage[english]{babel}

\usepackage[margin=2cm]{geometry}
\usepackage{graphicx}
\usepackage{mathtools}
\usepackage{amsmath,amsthm}
\usepackage{amsfonts,amssymb}
\usepackage{amscd}
\usepackage{amsthm}
\usepackage{verbatim}
\usepackage{tikz}
\usepackage[all,cmtip]{xy}
\usepackage{tikz-cd}
\usepackage{geometry}
\usepackage{makecell,xcolor}
\usepackage{float}
\usepackage{subcaption} 
\usepackage{eucal}
\usepackage{mathalfa}
\usepackage{mathrsfs}
\usepackage{fancyhdr}
\theoremstyle{plain}
\usepackage{setspace}
\tikzstyle{slice}=[text=gray]

\usepackage{xcolor}

\usepackage{enumitem}
\newtheorem{lemma}{Lemma}[section]

\newtheorem{theorem}[lemma]{Theorem}
\newtheorem{corollary}[lemma]{Corollary}

\theoremstyle{definition}

\newtheorem{definition}[lemma]{Definition}
\newtheorem{propos}[lemma]{Proposition}

{

}

\theoremstyle{remark}

\newtheorem{remark}[lemma]{Remark}

\newcommand{\Ext}{\mathop{\mathrm{Ext}}\nolimits}

\def\P{\mathbb P}
\def\A{\mathbb A}
\def\O{\mathcal{O}}

\newcommand{\modl}{\mathop{\mathrm{mod}}\nolimits}

\newcommand{\Hom}{{\mathrm{Hom}}}

\newcommand{\Coh}{{\mathrm{Coh}}}

\newcommand{\Mod}{{\mathrm{Mod}}}

\newcommand{\coker}{{\mathrm{coker}}}

\newcommand{\Rex}{{\mathrm{Rex}}}
\newcommand{\Lex}{{\mathrm{Lex}}}

\newcommand{\Fun}{{\mathrm{Fun}}}

\newcommand{\Ab}{{\mathrm{Ab}}}
\newcommand{\Eff}{{\mathrm{Eff}}}

\newcommand{\MCM}{{\mathrm{MCM}}}

\newcommand{\Bun}{{\mathrm{Bun}}}

\newcommand{\Modl}{{\mathrm{Mod}}}

\newcommand{\lex}{{\mathrm{lex}}}
\newcommand{\defl}{{\mathrm{def}}}

\newcommand{\xrightarrowdbl}[2][]{%
  \xrightarrow[#1]{#2}\mathrel{\mkern-14mu}\rightarrow
}

\def\kk{\mathrm{k}}

\def\Ee{\EuScript E}

\def\Aa{\EuScript A}

\title{A note on abelian envelopes}
\author{Anya Nordskova}
\address[Anya Nordskova]{Vakgroep Wiskunde, Universiteit Hasselt, Agoralaan gebouw D, 3590 Diepenbeek, Belgium}
\email{anya.nordskova@gmail.com}
\begin{document}
\begin{abstract}
This is a short note bridging the gap between two notions of universal abelian categories associated to exact categories, namely, Rump's quotient categories and Bodzenta-Bondal's abelian envelopes. The established connection allows us to draw several easy conclusions, in particular, we answer some (basic) questions raised by Bodzenta and Bondal in \cite{BB}. 
\end{abstract}
\maketitle

\section{Introduction}

In \cite{R20} Rump defines the \emph{left (resp. right) quotient category} $Q_l(\Ee)$ (resp. $Q_r(\Ee)$) for any exact category $\Ee$. More generally, he considers left (right) exact categories, which are defined by imposing only half of the axioms of the usual definition, see Section \ref{sec:prelims} for details. The category $Q_l(\Ee)$ comes with an exact full embedding $i_l: \Ee \to Q_l(\Ee)$ and is \emph{left abelian} (a natural ``one-sided'' generalisation of the notion of abelian categories). It was shown in \cite{R20} that $Q_l(\Ee)$ together with $i_l$ is universal among left abelian categories $\Aa$ and right exact functors $\Ee \to \Aa$, see Theorem \ref{thm:rump_univ} for the precise statement. In \cite{BB} Bodzenta and Bondal defined \emph{right (resp. left) abelian envelopes} $\Aa_r(\Ee)$ (resp. $\Aa_l(\Ee)$) of exact categories. By definition, a right abelian envelope $\Aa_r(\Ee)$ of an exact category $\Ee$ is an abelian category together with a right exact functor $\Ee \to \Aa_r(\Ee)$, satisfying a universal property for all abelian categories $\Aa$ and right exact functors $\Ee \to \Aa$ (Definition \ref{def:env}). Unlike Rump's quotient categories, abelian envelopes do not always exist. 

Even though the definitions look quite similar, the precise relation between the two notions is not immediately clear. For any exact category $\Ee$, it is evident (see \cite{BB}) from the two universal properties that $\Aa_r(\Ee)$ exists and coincides with $Q_l(\Ee)$ if the latter is abelian. The main conclusion of this note is that the converse also holds, i.e. if $\Aa_r(\Ee)$ exists then $Q_l(\Ee)$ is abelian and they coincide. We make this connection via the explicit description of right abelian envelopes obtained by Bodzenta and Bondal. Namely, they show that $\Aa_r(\Ee)$, if exists, is isomorphic to the subcategory of compact objects $\Lex(\Ee)^c$ of left exact contravariant functors from $\Ee$ to abelian groups $\Ab$. Here we observe that $\Lex(\Ee)^c$ is always left abelian and together with the Yoneda embedding $i: \Ee \to \Lex(\Ee)^c$ it satisfies the universal property of $Q_l(\Ee)$ for any exact category $\Ee$. It follows in particular that the universal functor $i_R: \Ee \to \Aa_r(\Ee)$ is always fully faithful and exact (not only right exact), which answers the questions in Remark 4.6 and after Example 4.7 in \cite{BB}. 

It is worth noting that although abelian envelopes of \cite{BB} essentially turn out to be a special case of Rump's quotient categories, this special case has many advantages, e.g. when considering derived categories (some of this is also discussed in \cite{R21}). We end the note with some simple applications and examples. 

\emph{Acknowledgements.}  I would like to express my gratitude to Michel Van den Bergh and Alexey Bondal for many interesting and helpful discussions. 

\subsection{Conventions.} Throughout the note we will assume that all categories are skeletally small.  All general ``one-sided'' (e.g. about right abelian envelopes) statements have natural dual counterparts which we do not mention.

\section{Preliminaries and Rump's quotient categories}\label{sec:prelims}

We recall the definitions of left exact and left abelian categories.

\begin{definition}
    A {\it left exact category} is an additive category with a distinguished class of cokernels ({\it deflations}) satisfying the following properties: 

    \begin{enumerate}
        \item  The class of deflations is closed under compositions and contains identity morphisms.
        \item The pullback of a deflation along any morphism exists and is again a deflation.
        \item If the composition $A \xrightarrow{f} B \xrightarrow{g} C$ is a deflation and $g$ has a kernel, then $g$ is a deflation.
    \end{enumerate}

Kernels of deflations will be referred to as {\it inflations}. A kernel-cokernel pair consisting of an inflation and a deflation is called a {\it conflation}.
\end{definition}

\begin{definition} An additive category $\mathcal{A}$ with cokernels is said to be {\it left abelian} if for every $f: A \to B$, $g: D \to B$ with $cg = 0$, $c := \coker(f)$, there exists a cokernel $d: E \to D$ such that $gd$ factors through $f$.

\begin{figure}[!h]
\centering
    \begin{tikzcd}
A \arrow[r, "f"]                   & B \arrow[r, "c", two heads]                                    & C \arrow[r] & 0 \\
E \arrow[r, "d"] \arrow[u, dashed] & D \arrow[u, "g"] \arrow[r] \arrow[ru, "0" description, dotted] & 0           &  
\end{tikzcd}
\end{figure}
\end{definition}

Dually one can define right exact and right abelian categories. Note that a category is abelian if and only if it is right and left abelian. An exact category in the sense of Quillen \cite{Q} is the same as a left and right exact category. 

\begin{definition} An additive functor $F: \Ee \to \mathcal{D}$ from a left exact category to an additive category is said to be {\it right exact} if $F(g) = \coker F(f)$ for every conflation ${A \xhookrightarrow{f} B \xrightarrowdbl{g} C}$  in $\Ee$. An additive functor $F: \Ee \to \mathcal{D}$ from a right exact category  to an additive category is {\it left exact} if ${F^{op}: \Ee^{op} \to \mathcal{D}^{op}}$ is right exact. 
\end{definition}

Let $\Ee$ be an additive category. By $\modl(\Ee)$ we denote the category of additive contravariant functors ${F: \Ee^{op} \to \Ab}$ admitting a presentation
$$ \Hom_{\Ee}(-,X) \xrightarrow{f} \Hom_{\Ee}(-,Y) \twoheadrightarrow F \to 0$$
for some morphism $f: X \to Y$ in $\Ee$. If $\Ee$ is a left exact category, denote by $\defl(\Ee)$ the full subcategory of $\modl(\Ee)$ of functors of the form $F = \coker(\Hom(-,f))$ with $f$ a deflation in $\Ee$. 

One can show that $\modl(\Ee)$ is a left abelian category and $\defl(\Ee)$ is a {\it thick} subcategory of $\modl(\Ee)$, i.e. a Serre subcategory closed under subobjects and epimorphic images, for any left exact category $\Ee$. The localization theory for left abelian categories developed by Rump \cite{R07}, \cite{R20} allows to define the quotient, which is again a left abelian category:

\begin{definition}[Rump, \cite{R20}]   The {\it left quotient category} $Q_l(\Ee)$ of $\Ee$ is defined by 
$$  Q_l(\Ee) := \modl(\Ee)/ \defl(\Ee)$$
    
\end{definition}

\begin{propos}[Rump, \cite{R20}, Proposition 5] The Yoneda embedding $\Ee \hookrightarrow \modl(\Ee)$ induces an exact full embedding $i: \Ee \hookrightarrow Q_l(\Ee)$ which reflects conflations. 
    
\end{propos}

Left quotient categories satisfy the following universal property. 

\begin{theorem}[Rump, \cite{R20}, Corollary 2 of Proposition 6]\label{thm:rump_univ}For any left exact category $\Ee$ and any left abelian category $\mathcal{A}$, the embedding $i: \Ee \hookrightarrow Q_l(\Ee)$ induces an equivalence of categories 
$$ (-) \circ i: \Fun_{\coker \mapsto \coker}(Q_l(\Ee),\mathcal{A}) \xrightarrow{\cong} \Rex(\Ee,\mathcal{A})$$

where by $\Fun_{\coker \mapsto \coker}$ we denote the category of additive functors respecting cokernels and by $\Rex$ the category of right exact functors. 
\end{theorem} 

\begin{remark} Note that $\Fun_{\coker \mapsto \coker}(\mathcal{B},\mathcal{C}) = \Rex(\mathcal{B},\mathcal{C})$ if $\mathcal{B}$ and $\mathcal{C}$ are abelian. 
\end{remark}

The key to proving this universal property is the fact that the left abelian category $Q_l(\Ee)$ together with the full embedding $i$ is a {\it dense extension} of $\Ee$. 

\begin{definition}[Rump, \cite{R20}, Definition 3]  Let $\Ee$ be a left exact category. A right exact full embedding of $\Ee$ into an additive category $i: \Ee \to \mathcal{A}$ is a {\it dense extension } if every object $A \in \mathcal{A}$ has a presentation $E_0 \xrightarrow{e} E_1 \xrightarrow{p} A \to 0$ with $E_0, E_1 \in \Ee$, $p = \coker(e)$, such that it satisfies the following two proprieties: 

\begin{enumerate}
    \item For every morphism $f: E \to A$ with $E \in \Ee$ there is a deflation $d: E' \twoheadrightarrow E$ in $\Ee$ such that $fd$ factors through $p$.
    \item For every morphism $f: E \to E_1$ such that $pf = 0$ there exists a deflation $d: E' \twoheadrightarrow E$ in $\Ee$ such that $fd$ factors through $e$.
\end{enumerate}

A presentation satisfying these conditions will be referred to as a {\it left exact presentation}. 
\end{definition}

Rump establishes the following fact which we will use later.

\begin{theorem}[Rump, \cite{R20}, Corollary 5 of Proposition 6]\label{thm:unique_extension} Let $\Ee$ be a left exact category. Then the left quotient category $\Ee \hookrightarrow Q_l(\Ee)$ is a dense extension of $\Ee$. Moreover, it is unique among dense extensions with cokernels in the following sense. If $i: \Ee \hookrightarrow \mathcal{A}$ is a dense extension of $\Ee$ with cokernels, then there is a unique equivalence $\varphi: \mathcal{A} \xrightarrow{\simeq} Q_l(\Ee)$ making the diagram commute: 
\begin{figure}[h!]
    \centering
\begin{tikzcd}
                                                                    & Q_l(\mathcal{E})                   \\
\mathcal{E} \arrow[ru, "Y_\mathcal{E}", hook] \arrow[r, "i"', hook] & \mathcal{A} \arrow[u, "\varphi "']
\end{tikzcd}
\end{figure}
\end{theorem}

\section{Connection to abelian envelopes of Bodzenta and Bondal}

Now we recall the definition of right abelian envelopes of exact categories.

\begin{definition}[Bodzenta-Bondal, \cite{BB}]\label{def:env} Let $\Ee$ be an exact category. A {\it right abelian envelope} of $\Ee$ is an abelian category $\Aa_r(\Ee)$ together with a right exact functor $i_R: \Ee \to \Aa_r(\Ee)$ such that for any abelian category $\mathcal{B}$ there is an equivalence of categories: 
$$(-)\circ i_R: \Rex(\Aa_r(\Ee),\mathcal{B}) \xrightarrow{\cong} \Rex(\Ee, \mathcal{B})$$
    
\end{definition}

Left abelian envelopes are defined dually. Unlike the quotient categories of Rump, right/left abelian envelopes do not always exist. 

\begin{definition} An object $X$ of a category $\mathcal{C}$ is \emph{compact} (also \emph{finitely presented} in the literature) if $Hom_\mathcal{C}(X,-)$ commutes with all filtered colimits. 
    
\end{definition}

The following facts were established in \cite{BB}: 

\begin{propos}[\cite{BB}, Lemma 4.5 and the discussion after the definition of envelopes] 
Let $\Ee$ be an exact category. 
\begin{enumerate}
    \item If the left quotient category $Q_l(\Ee)$ is abelian, then the right abelian envelope $\Aa_r(\Ee)$ exists and coincides with $Q_l(\Ee)$. 
    \item  Let $\Lex(\Ee) := \Lex(\Ee^{op}, \Ab)$ be the category of contravariant left exact functors from $\Ee$ to the category $\Ab$ of abelian groups. Denote by $\Lex(\Ee)^c$ the subcategory of compact objects in $\Lex(\Ee)$. If $\Aa_r(\Ee)$ exists, then it is equivalent to $\Lex(\Ee)^c$. Moreover, the functor $i_R$ is faithful. 
\end{enumerate}
\end{propos}

It is remarked, however, that the authors do not know whether the functor $i_R$ in 2) composed with the natural embedding of compact objects $ \Lex(\Ee)^c \hookrightarrow \Lex(\Ee)$ is the Yoneda embedding $Y_\Ee: \Ee \hookrightarrow\Lex(\Ee)$. Indeed, this follows from their proof only provided that $i_R$ is full, but this is also stated as unknown. In addition, it is natural to ask if $\Aa_r(\Ee)$ always exists when the category $\Lex(\Ee)^c$ is abelian. We answer these questions by establishing a more precise connection between Bodzenta-Bondal's abelian envelopes and Rump's quotient categories: 

\begin{theorem}\label{thm:connection} Let $\Ee$ be an exact category. Then the category $\Lex(\Ee)^c$ together with the Yoneda embedding $Y_\Ee$ is equivalent to the left quotient category $Q_l(\Ee)$. 
    
\end{theorem}

Note that it is claimed in particular that the Yoneda embedding $Y_\Ee: \Ee \to \Lex(\Ee)^c$ maps $\Ee$ to compact objects in $\Lex(\Ee)$ (this will be explained later). We conclude:  

\begin{corollary}\label{cor:env_exists} The following are equivalent for any exact category $\Ee$.  
\begin{enumerate}
    \item $\Aa_r(\Ee)$ exists.
    \item $Q_l(\Ee)$ is abelian. 
    \item $\Lex(\Ee)^c$ is abelian. 
\end{enumerate}
In this case the universal functor $i_R: \Ee \to \Aa_r(\Ee)$ is the Yoneda embedding $Y_\Ee: \Ee \to \Lex(\Ee)^c$. In particular, $i_R$ is exact and fully faithful. 
    
\end{corollary}

\begin{remark} \begin{enumerate}
    \item Theorem \ref{thm:connection} claims that $\modl(\Ee)/\defl(\Ee) \cong \Lex(\Ee)^c$ for any exact category $\Ee$. One can also show this directly. Let us sketch the argument in the case when $\modl(\Ee)$ is abelian (equivalently, when $\Ee$ has \emph{weak kernels}, see the discussion after Definition \ref{def:extker}). Recall that $\Lex(\Ee)$ is the quotient of the category of all additive contravariant functors $\Modl(\Ee)$ from $\Ee$ to $\Ab$ by the Serre subcategory of {\it weakly effaceable} functors $\Eff(\Ee)$, see e.g. \cite{Kel3}. By the discussion after Theorem \ref{thm:compactfp}, any compact object in both $\Modl(\Ee)$ and $\Lex(\Ee)$ is the cokernel of a morphism in $\Ee$. Hence the localisation functor, which is an equivalence on $\Ee$, induces an exact essentially surjective functor ${Q: \modl(\Ee) \cong (\Modl{\Ee})^c \to \Lex(\Ee)^c}$. It is known that $\Eff(\Ee) \cap \modl(\Ee) = \defl(\Ee)$ (e.g. \cite{HKvR}, Proposition 3.20 and \cite{S}, Lemma 9). Thus the kernel of $Q$ is $\defl(\Ee)$. 
    \item Let $\lex(\Ee)$ be the full subcategory of $\modl(\Ee)$ consisting of left exact functors. We warn the reader that $\lex(\Ee)$ is not necessarily equivalent to $\Lex(\Ee)^c$ even when $\Ee$ has weak kernels. To be more precise, in this case $\lex(\Ee)$ is the subcategory of {\it $\defl(\Ee)$-closed objects} in $\modl(\Ee)$, i.e. $\lex(\Ee) = \{ X \in \modl(\Ee) \ | \ \Hom(\defl(\Ee),X) = \Ext^1(\defl(\Ee),X) = 0 \}$ (see \cite{O}, Proposition 2.8). Hence, if we assume additionally that the localisation functor $L: \modl(\Ee) \to \modl(\Ee)/\defl(\Ee) $ has a right adjoint, then $\modl(\Ee)/\defl(\Ee) \cong \lex(\Ee)$. However, this is not always the case ([Ibid., Example 2.10]). Note that if $\Ee$ has enough projectives, then $L$ indeed has a right adjoint ([Ibid., Proposition 2.17]). 
    \item Note an immediately consequence of the fact that the universal functors $i_R: \Ee \to \Aa_r(\Ee)$, $i_L: \Ee \to \Aa_l(\Ee)$ are exact and not only right and left exact respectively. Namely, if $\Ee$ has both abelian envelopes, then the two universal properties give rise to a pair of adjoint functors between $\Aa_r(\Ee)$ and $\Aa_l(\Ee)$ (see also \cite{R21}). 
\end{enumerate}
\end{remark}

We will need some basic facts about Grothendieck categories and compact objects in the proof of Theorem \ref{thm:connection}. 

\begin{definition} A category $\mathcal{A}$ is {\it locally finitely presented} if it has filtered colimits, the full subcategory of compact objects $\mathcal{A}^c$ is skeletally small and every object of $\mathcal{A}$ is a filtered colimit of compact objects.  A {\it Grothendieck category} is an abelian category which has a generator and small colimits and satisfies (AB5), i.e. small filtered colimits of exact sequences are exact.
\end{definition}
\begin{remark} An abelian category is locally finitely presented if and only if it is a Grothendieck category with a generating set of compact objects (see \cite{B}, Satz 1.5, \cite{CB}, 2.4).
    
\end{remark}

\begin{theorem}[Breitsprecher, \cite{B}, Satz 1.11]\label{thm:compactfp} Let $\mathcal{C}$ be a locally finitely presented Grothendieck category and $\mathcal{S} \subset \mathcal{C}$ a class of compact generators. Then $X \in \mathcal{C}$ is compact if and only if there is a presentation $\bigoplus_{i=0}^n S_i \to \bigoplus_{i=0}^m S_i' \to X \to 0$ with $S_i, S_i' \in \mathcal{S}$. 
    
\end{theorem}

Let $\Ee$ be an exact category. It is well known that $\Lex(\Ee)$ is a Grothendieck category and the Yoneda functor $Y_\Ee: \Ee \to \Lex(\Ee)$, sending $E \in \Ee$ to $\Hom_{\Ee}(-,E)$, is fully faithful and exact. One can show that $\Ee$ is a class of compact generators of $\Lex(\Ee)$. The fact that representable functors are compact follows from the fact that they are compact in $\Mod(\Ee)$ and the subcategory $\Lex(\Ee)$ is closed under filtered colimits (see e.g. \cite{LK}). In particular, the Yoneda embedding $Y_\Ee$ maps $\Ee$ into $\Lex(\Ee)^c$ and every object of $\Lex(\Ee)^c$ is the cokernel of a morphism in $\Ee$ by {Theorem \ref{thm:compactfp}}. 

\begin{propos}[e.g. Lowen-Kaledin, \cite{LK}, Proposition 2.17]\label{thm:lowen_kaledin} Let $\Ee$ be an exact category. 
    Let $p: X \twoheadrightarrow E$ be an epimorphism in $\Lex(\Ee)$ with $E \in \Ee$. Then there exist $E' \in \Ee$ and a morphism $g: E' \to X$ in $\Lex(\Ee)$ such that the composition $fg: E' \twoheadrightarrow E$ is again an epimorphism. 
\end{propos}

The following result implies that $\Lex(\Ee)^c$ is left abelian for any exact category $\Ee$. 

\begin{theorem}[Breitsprecher \cite{B}, see also Crawley-Boevey \cite{CB}, Rump \cite{R10}]
    Let $\mathcal{A}$ be an additive locally finitely presented category. Then $\mathcal{A}$ is abelian if and only if $\mathcal{A}^c$ is left abelian. 
\end{theorem}

\begin{proof}[Proof of Theorem \ref{thm:connection}] Since the category $\Lex(\Ee)^c$ has cokernels, due to Theorem \ref{thm:unique_extension} it is sufficient to show that $Y_\Ee: \Ee \to \Lex(\Ee)^c $ is a dense extension. As discussed above, every $X \in \Lex(\Ee)^c$ has a presentation 
$$ E_0 \xrightarrow{a} E_1 \xrightarrow{p} X \to 0$$ 
with $E_0, E_1 \in \Ee$. One can show that this is always a left exact presentation, i.e. it satisfies the two conditions in the definition of a dense extension. Indeed, let $f: E \to X$ be a morphism with $E \in \Ee$. Consider the pullback of $p$ and $f$ in the ambient abelian category $\Lex(\Ee)$. Since $p$ is an epimorphism, so is its pullback $q: Y \to E$. By Proposition \ref{thm:lowen_kaledin} there exist $E' \in \Ee$ and a morphism $t: E' \to Y$ such that $d := qt$ is again an epimorphism. Observe that $d$ is the desired deflation. The second condition is easily established in the same fashion. 

\begin{figure}[h!]
    \centering
    \begin{tikzcd}
E_0 \arrow[r, "a"]                                        & E_1 \arrow[r, "p", two heads]              & X \arrow[r]       & 0 \\
E' \arrow[r, "t"] \arrow[rr, "d"', two heads, bend right] & Y \arrow[r, "q", two heads] \arrow[u, "g"] & E \arrow[u, "f"'] &  
\end{tikzcd}
\end{figure}

\end{proof}

\begin{remark}  It is also possible to show that $Y_\Ee: \Ee \to \Lex(\Ee)^c$ satisfies the universal property of Theorem \ref{thm:rump_univ} directly. The proof would be essentially the same as Rump's proof for $Q_l(\Ee)$ in \cite{R20}. 

\end{remark}

 Let us now discuss the conditions on $\Ee$ under which it possesses a right/left abelian envelope. 

\begin{definition}[Rump, \cite{R21}, Definition 3]\label{def:extker} Let $\Ee$ be a left exact category. A morphism $g:C \to B$ in $\Ee$ is said to be an {\it Ext-kernel} of $f: B \to A$ if $fg = 0$ and for any $g': C' \to B$ such that $fg' = 0$ there exists a deflation $d: D \to C'$ such $g'd$ factors through $g$.
  
    \begin{figure}[h!]
        \centering
        \begin{tikzcd}
C \arrow[r, "g"]                               & B \arrow[r, "f"]                                       & A \\
E' \arrow[r, "d", two heads] \arrow[u, dotted] & C' \arrow[u, "g'"] \arrow[ru, "0" description, dotted] &  
\end{tikzcd}
       
    \end{figure}

A left exact category $\Ee$ is said to be {\it left Ext-coherent} if every morphism in $\Ee$ has an Ext-kernel.  
\end{definition}

Dually one can define Ext-cokernels and right Ext-coherent categories. Recall that $g:C\to B$ is called a \emph{weak kernel} of $f:B \to A$ if $fg = 0$ and every $g'$ with $fg' = 0$ factors through $g$. An additive category with weak kernels is called {\it left coherent} (in the literature the name \emph{right coherent} is commonly used, but here we consistently follow the conventions of Rump). Unlike Ext-coherence, this property does not depend on an exact structure. Now thanks to the following result we can add another equivalent condition to the list in Corollary \ref{cor:env_exists}: 

\begin{theorem}[Rump, \cite{R21}, Proposition 6]\label{thm:extkers} Let $\Ee$ be an exact category. Then $Q_l(\Ee)$ is abelian (i.e. $\Aa_r(\Ee)$ exists) if and only if $\Ee$ is left Ext-coherent. 
\end{theorem}

Let us state a few immediate corollaries. 

\begin{theorem}\label{cor:props1} \begin{enumerate}
\item Let $(\Ee, \epsilon)$ be an exact category, where $\epsilon$ denotes the class of conflations. Let $(\Ee, \epsilon')$ be an exact category with the same underlying additive category $\Ee$, but a stronger exact structure, i.e. $\epsilon \subseteq \epsilon'$. If $(\Ee, \epsilon)$ possess  a right abelian envelope, then so does $(\Ee, \epsilon')$. 
    \item If $\Ee$ is left coherent, then $\Ee$ has a right abelian envelope (for the case of the split exact structure this is stated in \cite{BB}, Proposition 4.14).\footnote{The author has been made aware that A. Bondal and A. Pavlov have also independently obtained this statement.} 
    \item Let $\Ee$ be an additive category with the split exact structure. Then $\Ee$ has a right abelian envelope if and only if it is left coherent.
 \end{enumerate}
\end{theorem}

\begin{proof}
    \begin{enumerate}
        \item Indeed, this is an immediate consequence of Theorem \ref{thm:extkers}, since adding more deflations can only create more Ext-kernels. 
        \item A weak kernel is always an Ext-kernel, so left coherent categories are left Ext-coherent. Another way to see it is by observing that if $\Ee$ is left-coherent, then $Q_l(\Ee)$ is abelian as a quotient of the abelian category $\modl(\Ee)$ by a Serre subcategory. 
        \item Observe that for $\Ee$ endowed with the split exact structure, Ext-kernels are exactly the same as weak kernels.
    \end{enumerate}
\end{proof}

\section{Examples}

\begin{propos}
    
\begin{enumerate}
 
    \item Let $R$ be a commutative local Cohen-Macauley ring. Then the category of maximal Cohen-Macauley modules $\MCM(R)$ has both a right and a left abelian envelope for any exact structure. 
    \item Let $X$ be a noetherian scheme with a resolution property (i.e. every coherent sheaf is a quotient of a locally free one). Let $\Ee := \Bun(X)$ be the category of locally free sheaves on $X$, considered as a fully exact subcategory of $\Coh(X)$. The category $\Ee$ has a right abelian envelope, namely, $\Coh(X)$. The category $\Bun(X)$ endowed with any stronger exact structure (e.g. obtained by removing points from $X$) also possess a right abelian envelope. 
    \item Let $X$ be a smooth projective variety with $\dim(X) = 1$ or $2$. Then the category $\Bun(X)$ of locally free sheaves has a right abelian envelope when endowed with {\it any} exact structure. On the other hand, for $\dim(X) \geq 3$ this is not true in general, e.g. for $X = \P^3$. 
    \item Let $X$ be a smooth projective variety. Consider the subcategory $\Ee$ of $\Bun(X)$ consisting of bundles of the form $\bigoplus_{i} \O_X(i)^{d_i}$. One can show that $\Ee$, unlike $\Bun(X)$, is always left coherent, hence it has a right abelian envelope with any exact structure. In particular, $$\Aa_r(\Ee, \epsilon_{\text{split}}) = \modl(\Ee) \cong \modl_{fp}(R_X) \cong \Coh_{\kk^*}(\A X)$$

    where $R_X$ denotes the homogeneous coordinate ring of $X$ and $\Coh_{k^*}(\A X)$ the category of equivariant sheaves on the affine cone over $X$. Note that the right abelian envelope of $\Ee$ with the exact structure inherited from $\Coh(X)$ is again $\Coh(X)$ because any coherent sheaf is a quotient of a vector bundle from $\Ee$ (see \cite{BB}, Theorem 4.11). 
\end{enumerate}

\end{propos}

\begin{proof}
    \begin{enumerate}
        \item This follows from Theorem \ref{cor:props1}.2) and the fact that $\MCM(R)$ has weak kernels and cokernels. The latter is established by Holm in \cite{H}, Proposition 4.2, basing on Auslander-Buchweitz’s maximal Cohen–Macaulay approximations \cite{AB}. 
        \item  The fact that $\Aa_r(\Ee) = \Coh(X)$ when $\Ee$ is endowed with the geometric exact structure is Corollary 4.12 of \cite{BB}. The rest follows from Theorem \ref{cor:props1}.1). 
        \item For $\dim(X) = 1$ or $2$, the category $\Bun(X)$ has kernels, hence it possesses a right abelian envelope for any exact structure. On the other hand, the category $\Bun(\P^3)$ does not have weak kernels (Bondal-Pavlov, via private communication). Hence, by Theorem \ref{cor:props1}.3), it has no right abelian envelope when considered with the split exact structure. 
    \end{enumerate}
\end{proof}

\begin{remark}
    I would like to thank Alexey Bondal for drawing my attention to the fact that $\Bun(\P^3)$ does not have weak kernels. 
\end{remark}

\bibliography{bibs}
\bibliographystyle{amsplain}
\end{document}